\documentclass{amsart}
\usepackage{amssymb,amsmath,amsfonts,amsthm,textcomp,mathrsfs}
\usepackage[dvips,all]{xy}
\UseComputerModernTips

\newtheorem{thm}{Theorem}[section]
\newtheorem{prop}[thm]{Proposition}

\newtheorem{cor}[thm]{Corollary}

\newtheorem{qst}[thm]{Question}

\theoremstyle{definition}
\newtheorem{dfn}[thm]{Definition}
\newtheorem{eg}[thm]{Example}

\theoremstyle{remark}
\newtheorem{rmk}[thm]{Remark}

\renewcommand{\phi}{\varphi}

\newcommand{\onto}{\twoheadrightarrow}

\renewcommand{\hat}{\protect\widehat}

\renewcommand{\i}[1]{\mathfrak{#1}}

\newcommand{\m}{\i{m}}

\newcommand{\depth}{\mathop{\mathrm{depth}}\nolimits}

\newcommand{\ann}{\mathop{\mathrm{Ann}}\nolimits}
\newcommand{\ass}{\mathop{\mathrm{Ass}}\nolimits}

\newcommand{\pd}{\mathop{\mathrm{pd}}\nolimits}

\newcommand{\dirlim}{\varinjlim}

\newcommand{\ext}{\mathop{\mathrm{Ext}}\nolimits}

\renewcommand{\*}{\bullet}
\newcommand{\lto}{\mathop{\longrightarrow\,}\limits}

\title[Ideals with Large Projective Dimension]{A Family of Ideals with Few Generators in Low Degree and Large Projective Dimension}
\author{Jason McCullough}
\address{Department of Mathematics, Univeristy of California - Riverside, 202 Surge Hall, 900 University Ave., Riverside, CA 92521}
\email{jmccullo@math.ucr.edu}


\begin{document}

\begin{abstract} Stillman posed a question as to whether the projective dimension of a homogeneous ideal $I$ in a polynomial ring over a field can be bounded by some formula depending only on the number and degrees of the minimal generators of $I$.  More recently, motivated by work on local cohomology modules in characteristic $p$, Zhang asked more specifically if the projective dimension of $I$ is bounded by the sum of the degrees of the generators.  We define a family of homogeneous ideals in a polynomial ring over a field of arbitrary characteristic whose projective dimension grows exponentially if the number and degrees of the generators are allowed to grow linearly.  We therefore answer Zhang's question in the negative and provide a lower bound to any answer to Stillman's question.  We also describe some explicit counterexamples to Zhang's question including an ideal generated by 7 quadrics with projective dimension 15.
\end{abstract}

\maketitle

\section{Introduction}

Throughout this paper, let $K$ be a field of arbitrary characteristic.  We consider the following question first raised by Stillman: 

\begin{qst}[Stillman, {\cite[Problem 3.14]{PS}}]\label{Q1} Is there a bound, independent of n, on the projective dimension of ideals in $R = K[X_1,\ldots,X_n]$ which are generated by $N$ homogeneous polynomials of given degrees $d_1,\ldots,d_N$?
\end{qst}

Equivalently, given a polynomial ring $R$ in arbitrarily many variables and a graded free resolution of the form
\[0\to F_t \to \cdots \to F_2 \to F_1 \to F_0 (= R) \to R/I \to 0,\]
does the module $F_1~=~\bigoplus_{i = 1}^N R(-d_i)$ determine a bound on the length of the resolution $t$?  Here $R(-d)$ denotes a rank-one free module over $R$ generated in degree $d$.  Only partial answers to Stillman's question are known.

More recently, Zhang proposed that the projective dimension of an ideal in a polynomial ring generated by $N$ elements of degrees $d_1,\ldots,d_N$ is bounded by the sum of the degrees $\sum_{i=1}^N d_i$.  In this paper we show that there are ideals with projective dimension far exceeding this bound and thus provide lower bounds on any possible answer to Stillman's question.  More precisely, we produce a family of ideals in $m + n \frac{(m + d - 2)!}{(m-1)!(d-1)!}$ variables with $n + m$ generators of degree $d$ and with projective dimension equal to the number of variables.

\section{Preliminaries and Background}

  We will use $A$ to denote an arbitrary Noetherian commutative ring with identity and reserve $R$ for a polynomial ring the field $K$.  We denote by $\m$ the graded maximal ideal of $R$.  For a module $M$ over $R$, we denote the projective dimension of $M$ by $\pd(M)$.  Given a homogeneous ideal $I$ of $R$, by the projective dimension of $I$ we mean $\pd(R/I)$.  We also denote the length of the maximal regular sequence on $M$ contained in $\m$ by $\depth(M)$.  

Stillman's question is partially motivated by the work done on three-generated ideals, that is, ideals generated by three homogeneous polynomials.  A construction of Burch \cite{Burch} in the local case, extended by Kohn \cite{Kohn} in the global case, shows that there exist three-generated ideals of arbitrarily large projective dimension.  However, if this construction is applied in the polynomial ring case, as the projective dimension grows, the degrees of the generators are forced to grow as well, thus motivating Stillman's question. More specifically, Engheta shows in \cite{Engheta05} that the degrees of the generators in Burch's construction for an ideal of projective dimension $n$ will have degree $\frac{n}{2}$.

In \cite{Bruns} Bruns proves a stronger result showing that three-generated ideals have resolutions not only of arbitrarily long length, but also matching any free resolution.  To be precise, Bruns shows that given a free resolution 
\[F_\*: \cdots \to F_4 \to F_3 \to F_2 \to F_1 \to F_0 \to M \to 0\]
 of a finitely generated module $M$ over a ring $A$, there exists free modules and maps
\[F_3 \lto^{f_3} F_2' \lto^{f_2} A^3 \lto^{f_1} A\]
such that the new complex
\[F_\*': \cdots \to F_4 \to F_3 \lto^{f_3} F_2' \lto^{f_2} A^3 \lto^{f_1} A \to A/I \to 0\]
is exact and thus is a free resolution of a three-generated ideal $I$.  Again, however, if one applies this construction to an arbitrary ideal, this comes at the expense of the degrees of the generators of the corresponding three-generated ideal.

Engheta \cite{Engheta07} studied the case of 3 cubics and showed that the projective dimension of an ideal generated by 3 cubics is at most 36.  This bound is likely not tight as the largest known projective dimension of an ideal generated by three cubics is five.  The first such example was found by Engheta \cite{Engheta10}.  In Section \ref{examples} we also construct a simple example of an ideal generated by three cubics with projective dimension five.

Further motivating the study of Question \ref{Q1} is the fact that it is equivalent to the existence of a bound on the Castelnuovo-Mumford regularity of an ideal in a polynomial ring based purely on the number and degrees of the minimal generators of $I$.  (See Problem 3.15 \cite{PS}.)  The equivalence was proved by Caviglia.  See \cite{Engheta05} for a nice explanation of this argument.  

In more recent work, Zhang \cite{Zhang} conjectured that an upper bound for the projective dimension of an ideal $I$ with $N$ generators in degrees $d_1,\ldots,d_N$ is simply $\sum_{i = 1}^N d_i$.  His work involved computations of local cohomology modules in characteristic $p > 0$ and applications of the Frobenius morphism.  Recall that given a module $M$ of $R$, the $j^{\text{th}}$ local cohomology module $H^j_I(M)$ of $M$ with respect to an ideal $I$ is defined to be
\[H_I^j(M) = \dirlim_t \ext_R^j(R/I^t,M)\]
where the direct limit is taken over maps $\ext_R^j(R/I^t,M) \to \ext_R^j(R/I^s,M)$ induced by the natural surjections $R/I^s \onto R/I^t$ for $s \ge t$.
  
Zhang proved the following result:

\begin{prop}[Zhang, {\cite[Proposition 3]{Zhang}}] Assume $I = (f_1,\ldots,f_N)$ is an ideal of $R = K[X_1,\ldots,X_n]$ (assumed to be of characteristic $p > 0$) such that $\sum \deg f_i < n$.  Then $H_\m^0(H_I^j(R)) = 0$ for every maximal ideal $\m$.
\end{prop}

This result led Zhang to ask the following stronger questions:

\begin{qst}[Zhang, {\cite[Question 4]{Zhang}}]\label{Q2} Let $R = K[X_1,\ldots,X_n]$ be the ring of polynomials such that $K$ is any field, and let $\m$ be any maximal ideal.  Assume $I = (f_1,\ldots,f_N)$ is an ideal of $R$ such that $\deg f_1 + \cdots + \deg f_N < n$.  Is $H^0_\m(R/I) = 0$?
\end{qst}

\begin{qst}[Zhang, \cite{Zhang}]\label{Q3} Given a homogeneous ideal $I$ in $R = K[X_1,\ldots,X_n]$ with $N$ generators of degrees $d_1,\ldots,d_N$, is
\[\pd(R/I) \le \sum_{i = 1}^N d_i?\]
\end{qst}

Zhang went on to show (\cite[Proposition 5]{Zhang}) that a positive answer to Question~\ref{Q2} would imply a positive answer to Question~\ref{Q3}.  Thus a positive answer to either question would also answer Stillman's question and provide an explicit upper bound.  In the next section, we show that this upper bound fails in general, thereby giving negative answers to Question~\ref{Q2} and Question~\ref{Q3}.

A similar question was considered by Caviglia and Kummini in \cite{CK}.  They answered a question of Huneke as to whether the projective dimension of an ideal $I$ is bounded by the size of the monomial support, that is, by the number of monomials that appear as terms in a minimal set of generators of $I$.  They construct a family of ideals whose projective dimension grows exponentially with the size of the monomial spread, thus answering Huneke's question in the negative.  Our paper is similar in spirit to that of Caviglia and Kummini, however none of the ideals in their construction exceed the bound in Zhang's question.  To be precise, for chosen $n_i \ge 2$ for each $1 \le i \le d$, their construction yields ideals with projective dimension $\prod_{i=1}^d n_i$ but with generators whose degrees sum to 
\[d \prod_{i=1}^d n_i - \sum_{i = 1}^{n-1} n_1\cdots\hat{n_i}\cdots n_d > \prod_{i=1}^d n_i.\]

\section{Definition of a Family of Ideals and the Main Result}

Set $M_{m,d} = \frac{(m + d - 1)!}{(m-1)!d!}$. Recall that $M_{m,d}$ is the number of monomials of degree $d$ in $m$ variables.   This formula follows by considering counting monomials of degree $d$ in $m$ variables as an ``$m$-choose-$d$-with-replacement'' problem.  We use this number to make the following definition.

\begin{dfn} Fix integers $m,n,d$ such that $m \ge 1$, $n \ge 0$ and $d \ge 2$. Order the $M_{m,d-1} = \frac{(m + d - 2)!}{(m-1)!(d-1)!}$ monomials of degree $d-1$ over the variables $X_1,\ldots,X_m$ and denote the $i^{\text{th}}$ such monomial by $Z_i$.  
Let $p = M_{m,d-1}$ and set $R = K[X_1,\ldots,X_m,Y_{1,1},\ldots,Y_{p,n}]$ be a polynomial ring in $m + pn$ variables over $K$.   
We define  $I_{m,n,d}$ to be the ideal generated by the following $m + n$ degreed $d$ homogeneous polynomials: 
\[\left\{X_i^d\,|\,1 \le i \le m\right\} \cup \left\{\sum_{j = 1}^{p} Z_jY_{j,k}\,|\,1\le k \le n\right\}.\]
\end{dfn}

\begin{eg}
If we set $m = 2$, $n = 2$ and $d = 3$ and let 
\[R = K[X_1,X_2,Y_{1,1},Y_{2,1},Y_{3,1}, Y_{1,2}, Y_{2,2},Y_{3,2}],\] \
then
\[I_{2,2,3} = \left(X_1^3,X_2^3,X_1^2Y_{1,1} + X_1X_2Y_{2,1} + X_2^2Y_{3,1},X_1^2Y_{1,2}+X_1X_2Y_{2,2} + X_2^2Y_{3,2}\right).\]
By the following result, $\pd(R/I_{2,2,3}) = 8$.
\end{eg}

\begin{thm}\label{T} Fix integers $m \ge 1$, $n \ge 0$ and $d \ge 1$ and let $p = M_{m,d-1}$.  Let $R = K[X_1,\ldots,X_m,Y_{1,1},\ldots,Y_{p, n}]$ and let $I = I_{m,n,d}$.  Then
\[\pd(R/I) = m + n p = m + n \frac{(m + d - 2)!}{(m-1)!(d-1)!}.\]
\end{thm}

\begin{proof} Let $\m$ be the graded maximal of R generated by all $m + np$ variables.  By the Auslander-Buchsbaum formula (see \cite[Theorem 19.9]{Eisenbud}), we have that 
\[\pd(R/I) = \depth(R) - \depth(R/I) = m + np - \depth_\m(R/I).\]
Thus proving that $\pd(R/I)$ has a maximal value of $m + np$ is equivalent to proving that $\depth(R/I) = 0$.  Thus it suffices to show that $\m \in \ass_R(R/I)$, which is equivalent to showing that $\m = \ann_{R/I}(T) = (I:_R T)$ for some $T \in R$. 

Let $T = \prod_{i = 1}^m X_i^{d - 1}$.  Note that $T \notin I$, and so $\m \supset (I:_{R} T)$.  It is clear that $T X_i \in I$ for $1 \le i \le m$.  

We now show that $T Y_{j,k} \in I$ for $1 \le j \le p$ and $1 \le k \le n$.
Note that the monomial $Z_i$ divides $T$ for every $1 \le j \le p$.  Write $S_j Z_j = T$ where $S_j$ is a monomial of degree $(d-1)^m - (d - 1)$. Thus we have modulo $I$:
\[ T Y_{j,k} = S_j Z_j Y_{j,k} = S_j \left(- \sum_{q \neq j} Z_q Y_{q,k} \right) = - \sum_{q \neq j} S_j Z_q Y_{q,k}.\]
For $m = 1$, this is an empty sum equaling $0$ and so clearly $TY_{j,k} \in I$.  If $m > 1$, observe that for $q \neq j$, $S_q Z_j$ must be divisible by $X_i^d$ for some $i$.  Therefore $T Y_{j,k} \in I$ for all $1 \le j \le p$ and $1 \le k \le n$.  
\end{proof}

\section{Examples and Asymptotic Growth of Projective Dimension}\label{examples}

In this section we describe explicitly some specific examples of the ideal construction from the previous section.  Later we consider the asymptotic growth of the projective dimension of a subfamily of the ideal $I_{m,n,d}$.

\begin{eg}[Three cubics with projective dimension 5]\label{example}
First we note that we recover another example of 3 cubics of projective dimension 5.  Let $R = K[X_1,X_2,Y_1,Y_2,Y_3]$ and let $I = I_{2,1,3}$.  (Since $n = 1$, we omit the second index on the $Y$-variables for simplicity.)  Then 
\[I = (X_1^3, X_2^3, X_1^2Y_1 + X_1X_2Y_2 + X_2^2Y_3).\]
  By the above theorem, 
\[\pd(R/I) = 2 + 1\left(M_{2,2}\right) = 2 + \frac{3!}{2!1!} = 2 + 3 = 5.\]
The first example of an ideal generated by 3 cubics with projective dimension 5 was found by Engheta \cite{Engheta10}.
\end{eg}

Since $M_{2,d}$ grows linearly with $d$, it takes at least 4 generators to find a counterexample to Zhang's question.  The following are three of the `smallest' counterexamples.

\begin{eg}[Seven quadrics with projective dimension 15] Let 
\[R = K[X_1,X_2,X_3,Y_{1,1},\ldots,Y_{3,4}],\]
 a polynomial ring with $15$ variables, and let $I = I_{3,4,2}$.   Then $I$ is generated by $7$ homogeneous polynomials of degree $2$; namely,
\[ I = (X_1^2,X_2^2,X_3^2,f_1,f_2,f_3,f_4)\]
where
\[
f_i = X_1Y_{1,i} + X_2Y_{2,i} + X_3Y_{3,i}
\]
for $i = 1$  to $4$.  Then Theorem~\ref{T}, $\pd(R/I) = 15$ while the sum of the degrees of the generators is $14$.
\end{eg}

\begin{eg}[Six cubics with projective dimension 20] Let 
\[R = K[X_1,\ldots,X_5,Y_1,\ldots,Y_{15}],\] a polynomial ring with $20$ variables, and let $I = I_{5,1,3}$.  (Again, since $n = 1$, we omit the second subscript on the $Y$ variables.)  Then $I$ is generated by $6$ homogeneous polynomials of degree $3$; namely,
\[ I = (X_1^3,X_2^3,X_3^3,X_4^3,X_5^3,f)\]
where $f$ is a cubic with $15$ terms.  Note that $I$ is an almost complete intersection.  Then by the Theorem \ref{T}, $\pd(R/I) = 20$.
\end{eg}

\begin{eg}[Four septics with projective dimension 31] Let 
\[R = K[X_1,X_2,X_3,Y_1,\ldots,Y_{28}],\]
and let $I = I_{3,1,7}$.    Then $I$ is generated by $X_1^7,X_2^7,X_3^7$ and one homogeneous polynomial of degree $7$ with $28$ terms.  By Theorem~\ref{T}, $\pd(R/I) = 31$.
\end{eg}

Note that as $m$, $n$, and $d$ grow, $\pd(R/I_{m,n,d})$ grows very quickly.  For instance, $I_{5,6,5}$ is an ideal with 11 generators in degree $5$ and has projective dimension 425.  Let us consider the case where the number of generators is fixed. When $m = 2$ and $n = 1$, we get that $\pd(R/I_{2,1,d}) = d + 2$.  Thus we produce ideals with three degree $d$ generators each with projective dimension $d + 2$.  Huneke and Eisenbud have shown (unpublished work) that the projective dimension of of an ideal generated by 3 quartics ($d = 2$ case) is at most $4$.  Engheta's example of three cubics with projective dimension $5$ (along with our Example~\ref{example}) is the largest known projective dimension of a three-generated ideal for the case $d = 3$.  The author does not know of any ideal with three degree $d$ generators with projective dimension greater than $d + 2$, so we ask the following question:

\begin{qst} Is the projective dimension of an ideal $I$ with 3 generators in degree $d$ bounded above by $d + 2$?
\end{qst}

More generally, for a fixed number of generators $N$ in degree $d$, our formula for the projective dimension of $I_{m,n,d}$ is a polynomial in $d$ of degree $N-2$, so we wonder:

\begin{qst} Is the projective dimension of an ideal $I$ with $N$ generators in degree $d$ bounded above by a polynomial in $d$ of degree $N - 2$?
\end{qst}

Finally, we consider the growth of $\pd(R/I_{n-1,1,n})$, where we allow both the number of generators and the degree to grow linearly.  Note that $I_{n-1,1,n}$ is an ideal with $n$ generators each in degree $n$.  By Theorem~\ref{T} its projective dimension is
\[\pd(R/I_{n-1,1,n}) = n-1 + \frac{(2n - 3)!}{(n-2)!(n-1)!} > \frac{2n-3}{n-2} \cdot\frac{2n-4}{n-3}\cdot\frac{2n-5}{n-4}\cdots\frac{n-2}{1} > 2^{n-2}.\]
We summarize this calculation in the following result:

\begin{cor} Any general upper bound for the projective dimension of an ideal with $N$ generators of degree $N$ must be at least $2^{N-2}$.
\end{cor}

\begin{rmk} A finer analysis of the asymptotic growth of $\pd(R/I_{n-1,1,n})$ is possible using Stirling's approximation.  (See e.g. \cite{GKP}.)  One finds that $\pd(R/I_{n-1,1,n})$ grows asymptotically as $\frac{ 4^{n-1}}{2\sqrt{\pi (n-1)}}$.
\end{rmk}

\section{Acknowledgments}

The author would like to thank Bahman Engheta, Bart Snapp, and Daniel Zaharopol for useful conversations and David Rush and Louis Ratliff for reading an earlier draft of this paper.  Early examples of the ideals defined in this paper were found with computations in Macaulay2 \cite{M2}.

\bibliographystyle{amsalpha}
\bibliography{bib}

\end{document}